\documentclass[microtype]{gtpart}

\usepackage[utf8]{inputenc}
\usepackage{amsmath}
\usepackage{amsfonts}
\usepackage[utf8]{inputenc}
\usepackage{graphicx}
\usepackage{setspace}
\usepackage{listings}
\usepackage{amsmath}%
\usepackage{amsthm}
\usepackage{MnSymbol}%
\usepackage{amsfonts}
\usepackage{eucal}

\usepackage{wasysym}%
\newtheorem{theorem}{Theorem} [section]  
\newtheorem{problem}{Problem}
\newtheorem{lemma}[theorem]{Lemma}          
    
\newtheorem{proposition}[theorem]{Proposition}
\theoremstyle{definition}
\newtheorem{definition}{Definition}[section]    
\newtheorem*{remark}{Remark}             
\newtheorem*{solution}{Solution}         
\newtheorem{conjecture}{Conjecture}
\usepackage{pinlabel}  
\usepackage{graphicx}  
\usepackage[all]{xy}
\usepackage{amscd}

\providecommand{\keywordss}[1]
{
  \small	
  \textbf{\textit{Keywords --- }} #1
}

\usepackage[english,serbian]{babel}
\title{About a sequence of points and a relationship between pencils of conics and circles in the Euclidean plane}
\author{Andrija Živadinović \newline Veljko Toljić}
\givenname{Andrija}
\surname{Zivadinovic}
\email{zivadinovic.andrija@gmail.com}
\givenname{Veljko}
\surname{Toljic}
\address{Gimnazija "Svetozar Marković" Niš, Republic of Serbia}
\email{tolja003@gmail.com}

\date{June 2021}

\begin{document}

\selectlanguage{english}

\begin{abstract}
   For a given triangle $\triangle ABC$, we define two sequences of points on line $BC$ and provide their generalizations to real functions such that centers of circumscribed circles around $A$ and adjacent points in subsequences generate a pencil of conics touching perpendicular bisectors of $AB$ and $AC$.
\end{abstract}

\maketitle
\keywordss{analytic geometry, locus constructions, conic sections, pencil of conics}

\section{Introduction }
The properties of conics regarding incidence have been extensively researched due to their applicability and are generally easier to tackle using projective geometry. However, metric properties, which are closely tied to the Euclidean plane have no better tools in general than analytic geometry. We provide a relationship between pencils of conics and circles.

All definitions and notation not introduced by authors can be found in \cite{2}

In a triangle $\triangle ABC$, trivially, the foot of a median divides $BC$ into two parts of equal lengths. Also well known is that the foot of the angle bisector divides $BC$ into two parts whose quotient is equal to $(\frac{AB}{AC})^{1}$. Similar holds for the foot of symedian: the quotient of cuts is equal to $(\frac{AB}{AC})^{2}$. If one observes carefully, 1 is the same as $(\frac{AB}{AC})^{0}$, so a pattern emerges. We can extend this sequence of points. 
\begin{definition}
Let $M_{k}$ for $k\in \mathbb{Z}$ be a point on segment $BC$ such that $$BM_{k}:M_{k}C=\left (\frac{AB}{AC}\right)^{k}$$

\end{definition}
\begin{definition}
Let $M'_k$ for $k\in \mathbb{Z}\backslash\{0\}$ be a point on the line $BC$ not lying on the segment such that $$BM'_{k}:M'_{k}C=\left (\frac{AB}{AC}\right)^{k}$$
\end{definition}
\begin{remark}
$M'_1$ is the foot of exterior angle bisector. $M'_2$ is the center of Apollonian circle of $\triangle ABC$ with respect to $BC$.
\end{remark}
\begin{remark}
As it turns out, these sequences are easy to construct, though the proof of the construction is not connected to the topic and is therefore left as an exercise for the reader:
\begin{enumerate}
    \item Given $M_i$ (or $M'_i$), we reflect it over $M_0$ to get $M_{-i}$ (or $M'_{-i}$)
    \item Given $M_{-i}$ (or $M'_{-i}$), we reflect line $AM_{-i} $ (or $AM'_{-i}$) over the interior angle bisector of $\angle BAC$ and in intersection with $BC$ we get $M_{i+2}$ (or $M'_{i+2}$)
\end{enumerate}
\end{remark}
An interesting thing happens when for each $i$ , one draws center $O$ of the circumscribed circle of triangle $\triangle AM_iM_{i+1}$. As it will be proven later, all points $O$ lie on a conic section. The same holds for $M'$ (Figure 1).

  \begin{figure}[htbp]
  \centering

  \includegraphics[scale=0.4]{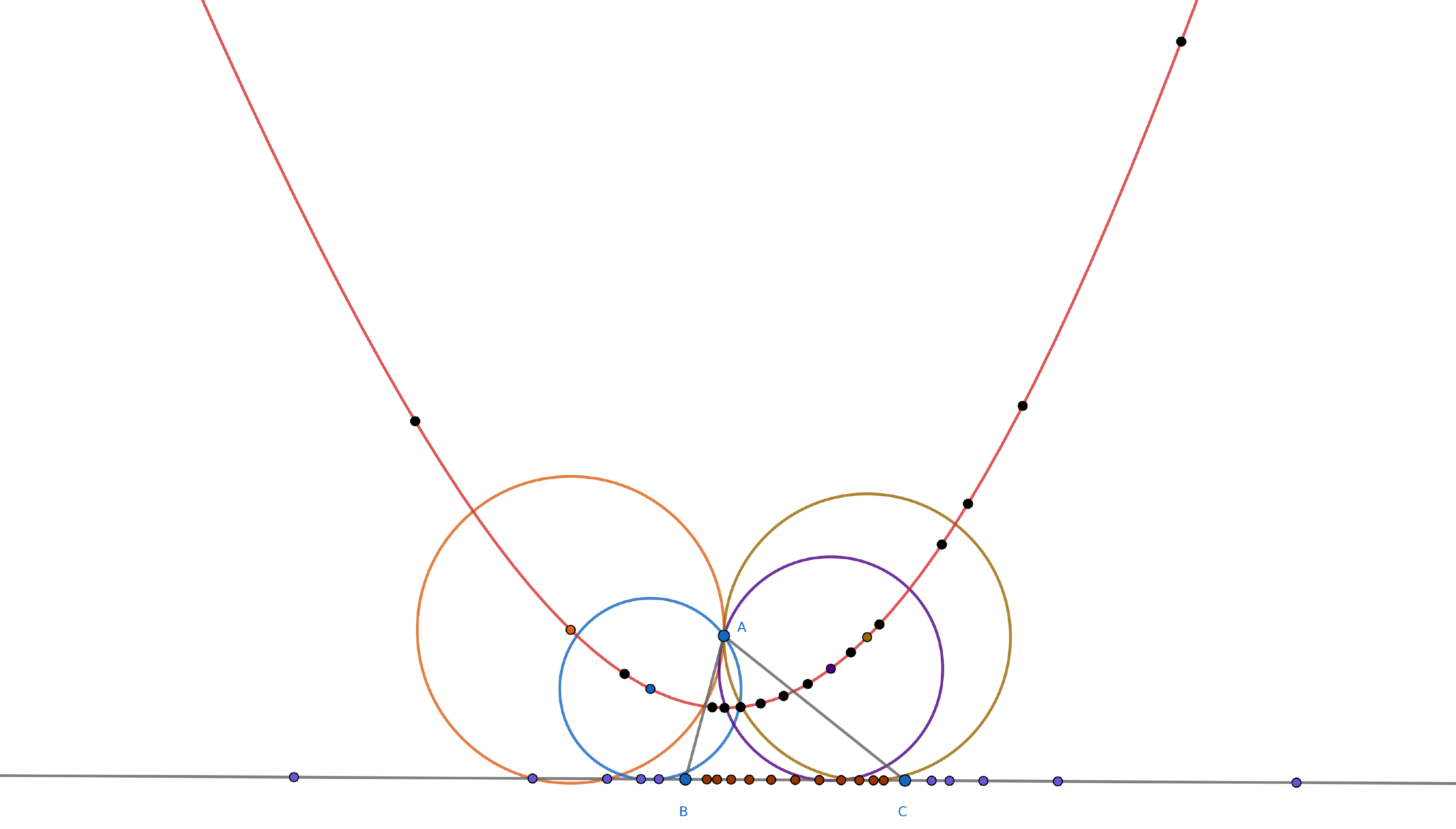}
  \caption{A conic built by centers of circumscribed circles}
 \end{figure}
Much more surprisingly, for every $k$, the centers of $\triangle AM_iM_{i+k}$  form a conic too and all these conics have 2 common points. (Figure 2)

We call those points $Z$ and $V$. Now we will extend the definition and state the main theorem.

  \begin{figure}[htbp]
  \centering

  \includegraphics[scale=0.4]{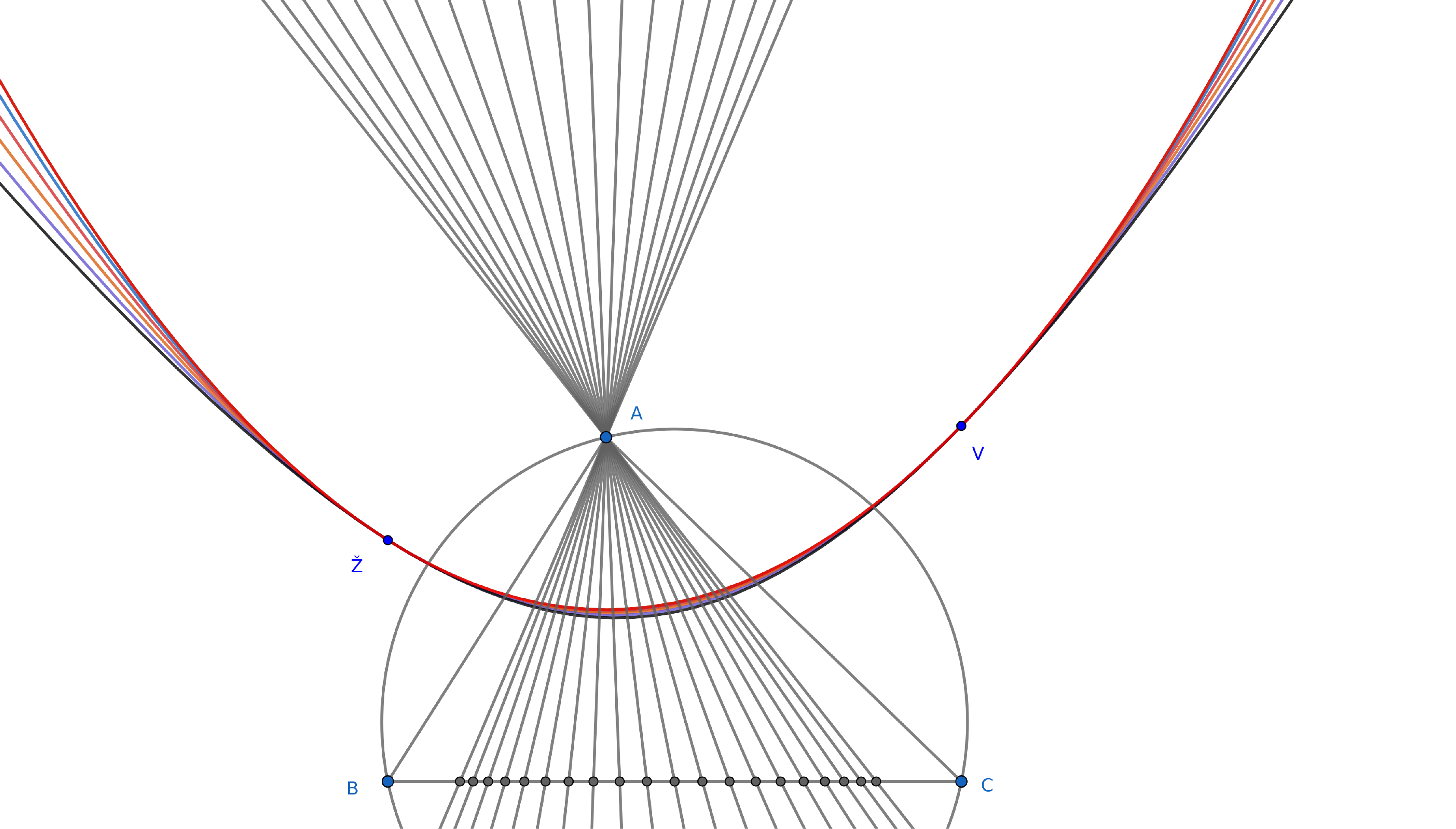}
  \caption{More conics}
 \end{figure}
 
 \newpage
\section{Main results}




%


\begin{definition}
For edge $BC$ of $\triangle ABC$ we define a mapping $M_a: \mathbb{R} \xrightarrow{} BC$ such that for $x \in \mathbb{R}$ $M_a(x)$ is a point on $BC$ such that $$BM_a(x):M_a(x) C=\left (\frac{AB}{AC}\right)^{x}$$
\end{definition}
\begin{definition}
For line $BC$ not including the edge, and for $x\in \mathbb{R}, x\neq 0$, we define $M'_a(x)$ analogously as $M_a(x)$.
\end{definition}
\begin{theorem}\label{thm1}
Let $ABC$ be a triangle, $AC\neq AB$. Let $r$ and $s$ be lines perpendicular to $BC$ containing $B$ and $C$, respectively. Let $Z$ and $V$ be intersections of $r$ and the perpendicular bisector of $AB$ and $s$ and perpendicular bisector of $AC$, respectively.
\begin{enumerate}
    \item For any fixed $t\in \mathbb{R}$, for every $k\in \mathbb{R}$, circumcenter of triangle $\triangle AM_a(k)M_a(k+t)$ lies on a fixed non-degenerate conic section. Let that curve be $g$. Then $t$ is called the \textit{span} of the curve $g$ and we write $Sp(g)=t$.
    \item For any fixed $t\in \mathbb{R}$, for every $k\in \mathbb{R}$, $k\notin \{0,-t\}$, circumcenter of triangle 
    
    $\triangle AM'_a(k)M'_a(k+t)$ lies on a fixed non-degenerate conic section.  Then $t$ is called the \textit{span} of the curve $g$ and we write $Sp(g)=t$.
    \item For any fixed $t\in \mathbb{R}\backslash\{0\}$, for every $k\in \mathbb{R}$, $k\neq 0$ , circumcenter of triangle $\triangle AM'_a(k)M_a(k+t)$ lies on a fixed non-degenerate conic section. Then $t$ is called the \textit{span} of the curve $g$ and we write $Sp(g)=t$.
    \item $Z$ and $V$ lie on every curve described in previous statements such that the tangent to every curve in $Z$ and $V$ are perpendicular bisectors of $AB$ and $AC$ respectively.
    \item If curve $g$ is described in ($i$), for $i\in\{1,2,3\}$, we say that $g$ is from $(i)$.If $g$ and $h$ are both from (1), (2) or (3) and $|Sp(g)|=|Sp(h)|$, then $g=h$. If $g$ is from (1) , $h$ is from (2) and $Sp(h)=Sp(g)$, then $g=h$.
 \end{enumerate}
\end{theorem}
Let $\mathbb{F}$ be the pencil of conics touching perpendicular bisector of $AB$ in $Z$ and perpendicular bisector of $AC$ in $V$.
\begin{theorem}\label{thm2}Every hyperbola or parabola in $\mathbb{F}$ can be obtained in a way described in Theorem \ref{thm1}, which means that the circle through $A$ around any point on them cuts $BC$.
\end{theorem}
In the section Applications, we will extend this statement even further, describing and classifying the whole pencil of conics.
\begin{remark}
$ZV$ is perpendicular to $AM_0$, where $M_0$ is the center of $BC$
\end{remark}
\begin{remark}
$ZV\cap BC$ is the center of the Apollonian circle of $\triangle ABC$ with respect to $BC$, or $M'_1$
\end{remark}
\section{Proof of Theorem \ref{thm1}}
For simplicity in notation, we will first prove Theorem \ref{thm1}(1) as if we were dealing with discrete $M$ and $M'$. As we will see, the proof is not dependent of $k$ and $t$ being integers.

We provide an analytic proof of Theorem \ref{thm1}(1). We set the problem into an Cartesian coordinate system with center in $B$ such that point $C$ has coordinates $(1,0)$, therefore the coordinates of $A$ are free, let $A\;(x_A,y_A)$. Let $c$ be $|AC|$  , $c$ be $|AB|$ and let $a$ be $|BC|$.

We will now calculate coordinates of point $M_{k}$, for some $k\in\mathbb{R}$. Because line $BC$ is the x-axis, we know that y coordinate of point $M_{k}$ is zero. For calculating x coordinate we will use these two expressions:
$$|M_{k}C| = \left(\frac{b}{c}\right)^{k} \cdot |BM_{k}| $$
$$|BM_{k}| + |M_{k}C| = |BC| = 1$$
\newpage
From this we can get:
$$|BM_{k}| \cdot (1 +  \left(\frac{b}{c}\right)^{k}) = 1$$
or
$$|BM_{k}| = \frac{c^{k}}{c^{k}+b^{k}}$$
We defined that point $B$ is the center of the coordinate system, and that line $BC$ is the x-axis, so $|BM_{k}|$ is the x coordinate of point $M_{k}$.

  \begin{figure}[htbp]
  \centering

  \includegraphics[scale=0.35]{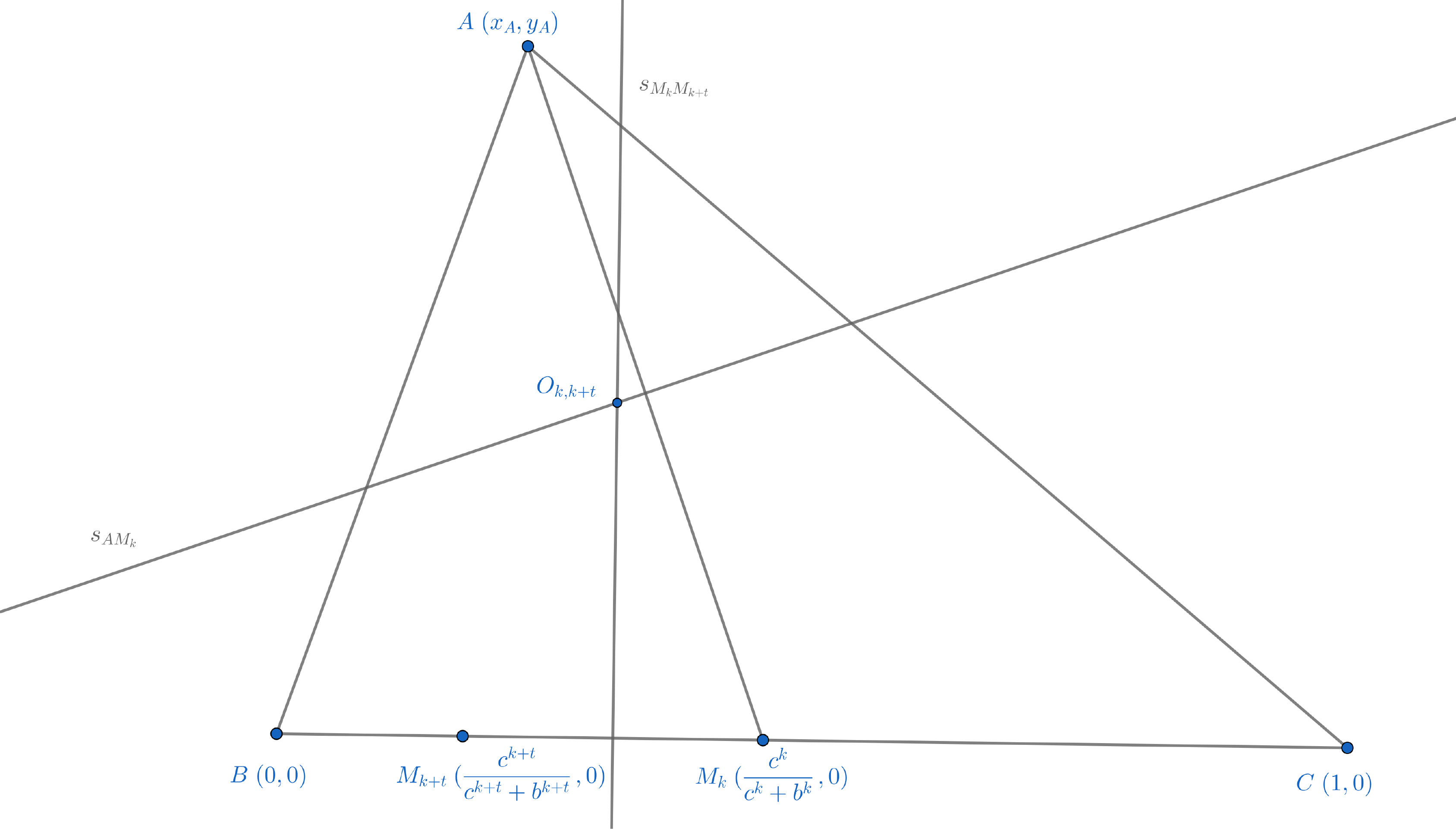}
  \caption{Analytic proof of Theorem \ref{thm1}(1)}
 \end{figure}

Let $O_{k,k+t}\;(x_O,y_O) $ be the circumcenter of triangle $\triangle AM_{k}M_{k+t}$. We know that point $O_{k,k+t}$ lies on perpendicular bisector of segment $M_{k}M_{k+t}$, so its x coordinate is the arithmetic mean of x coordinates of points $M_{k}$ and $M_{k+t}$, or:
\begin{equation}
    \label{eq:1}
    x_{O} = \frac{\frac{c^{k}}{c^{k}+b^{k}} + \frac{c^{k+t}}{c^{k+t}+b^{k+t}}}{2}
\end{equation}
Let's assume that $AM_k$ is not the height of the triangle(that case will be done separately).
The slope of the line $AM_k$ is: 
$$k_{AM_k}=\frac{y_{A}}{x_{A}-\frac{c^{k}}{c^{k}+b^{k}}}$$
Let $S$ be the midpoint of $AM_k$. Then the coordinates of $S$ are :
$$S\; (\frac{x_A+\frac{c^{k}}{c^{k}+b^{k}}}{2},\frac{y_A}{2})$$
The slope of $k_s$ of the perpendicular bisector of $AM_k$ is:
$$k_s=\frac{\frac{c^{k}}{c^{k}+b^{k}}-x_{A}}{y_{A}}$$
We can now easily calculate the constant term $n_{s}$ of the perpendicular bisector of $AM_k$:
$$n_s=\frac{x_A^2+y_A^2-\left(\frac{c^{k}}{c^{k}+b^{k}}\right)^2}{2y_A}$$
The following equation holds for coordinates of $O_{k,k+t}$:
\begin{equation}\label{eq:2}
y_O=\frac{\frac{c^{k}}{c^{k}+b^{k}}-x_{A}}{y_{A}}\cdot x_O+\frac{x_A^2+y_A^2-\left(\frac{c^{k}}{c^{k}+b^{k}}\right)^2}{2y_A}
\end{equation}

Here comes the tricky part: We will label $\frac{c^{k}}{c^{k}+b^{k}}$ with $m$ and $\frac{c^{k+t}}{c^{k+t}+b^{k+t}}$ with $n$.

We can now get $$\frac{1}{m}=1+\left(\frac{b}{c}\right)^{k}$$
and 
$$\frac{1}{n}=1+\left(\frac{b}{c}\right)^{k+t}$$

Now we can get the relationship between $m$ and $n$ by canceling out the parametric part $\left(\frac{b}{c}\right)^{k}$ :

$$\frac{\frac{1}{n}-1}{\frac{1}{m}-1}=\left(\frac{b}{c}\right)^{t}$$
By further simplification we get:
\begin{equation}\label{eq:left1}
    n=\frac{c^t\cdot m}{(c^t-b^t)\cdot m+b^t}
\end{equation}
From now on we denote $p$ to be $b^t$ and $l$ to be $c^t$.

Equation \eqref{eq:1} now becomes:

\begin{equation}\label{eq:3}
    x_O=\frac{m+\frac{l\cdot m}{m(l-p)+p}}{2}
\end{equation}
and equation \eqref{eq:2} becomes:
\begin{equation}\label{eq:4}
y_O=\frac{m-x_{A}}{y_{A}}\cdot x_O+\frac{x_A^2+y_A^2-m^2}{2y_A}
\end{equation}

Now we see that equation \eqref{eq:3} is quadratic in terms of $m$. We can solve it and substitute solution for $m$ into equation \eqref{eq:4}. That is how we get a relationship between $x_O$ and $y_O$ which is not parametric, and see that it is a conic equation.

Quadratic equation derived from \eqref{eq:3}:

\begin{equation}\label{eq:5}
    (l-p)\cdot m^2 + (2x_O\cdot(p-l)+p+l)\cdot m - 2x_Op
\end{equation}

Solutions for $m$ from equation \eqref{eq:5}:
$$m_{1,2}=\frac{2 (l -p)x_{O} - l - p \pm \sqrt{4 l^{2} x_{O}^{2} - 4 l^{2} x_{O} + l^{2} - 8 l p x_{O}^{2} + 8 l p x_{O} + 2 l p + 4 p^{2} x_{O}^{2} - 4 p^{2} x_{O} + p^{2}}}{2(l - p)}$$
We used sympy powered program to do the dirty work. We will provide the link to the program at the end of the paper: \cite{link}. This is the solution where $l\neq p$, the case $p=l$ will be handled separately.

First, we labeled the square root of the discriminant of the quadratic equation with $D$. Then we used the quadratic equation to get rid of $m^2$ in equation \eqref{eq:4}. When we got the equation in terms of linear $m$, we substituted in $D$ and solved for it. Note that the $\pm$ will disappear since we are squaring, but therefore we lost the equivalence, which doesn't mean very much. Then we substitute the value of the discriminant back and simplify the equation. This is how we obtain the conic equation:
    \begin{equation}\label{eq:6}
        4  y_{O}^{2} y_{A}^{2}\left(p-l\right)^{2}
    \end{equation}
    $$+ x_{O}  y_{O} y_{A} \left(p-l\right)^{2} \left(8 x_{A} - 4\right)$$ 
    $$ + x_{O}^{2} \left(\left(p-l\right)^{2} \left(2 x_{A} - 1\right)^{2}- \left(p+l\right)^{2}\right) $$
    $$- 2y_{O} y_{A}\left(2 c^{2} \left(p-l\right)^{2}- \left(p+l\right)^{2}\right)$$
    $$+ x_{O} \left(- 2 c^{2} \left(p-l\right)^{2} \left(2 x_{A} - 1\right) + 2 x_{A} \left(p+l\right)^{2}\right)$$ 
    $$+c^{2} \left(c^{2} \left(p-l\right)^{2} - \left(p+l\right)^{2}\right)=0$$

\newpage
And now we will cover corner cases:

If $AM_k$ is the height of the triangle, then we know two things: x-coordinate of $M_k$ is $x_A$ and the y-coordinate of $O_{k,k+t}$ is $\frac{y_A}{2}$. Because the relationship between $n$ and $m$, equation \eqref{eq:left1}, still holds, we can calculate $x_O$ by substituting $n$ and $m$ into \eqref{eq:1}:

$$x_O=\frac{x_A+\frac{l\cdot x_A}{(l-p)x_A+p}}{2}$$
We now substitute coordinates of $O_{k,k+t}$ into \eqref{eq:6} (using computer this is easy).

Another corner case comes from the fact that we, through the solving of quadratic equation \eqref{eq:5} we divided by $p-l$, so the case $p=l$ is not yet covered. For this, in the statement of the problem, we assumed $c\neq b$, so $p$ can be equal to $l$ if and only if $t=0$. When $t=0$, all our points are equidistant from $A$ and the line $BC$, so by definition, they lie on a parabola with focus $A$ and directrix $BC$. 

Now for the second statement of the theorem, we will calculate the coordinates of the circumcenter and plug them into \eqref{eq:6}. (This will also prove the second sentence in the fifth statement). Analogously to the point $M_k$ we get that the x-coordinate of $M'_k$ is $\frac{c^k}{c^k-b^k}$.Let $O_{k,k+t}$, again, denote the center of the circumscribed circle around $\triangle AM'_kM'_{k+t}$ To get the x-coordinate of $O_{k,k+t}$, we , analogously to the first statement, calculate the arithmetic mean of x-coordinates of $M'_k$ and $M'_{k+t}$:
$$x_O=\frac{\frac{c^{k} l}{ c^{k} l- b^{k} p} + \frac{c^{k}}{c^{k} - b^{k}}}{2}$$

Similarly, the y-coordinate of $O_{k,k+t}$ is :

$$y_O=\frac{\frac{c^{k}}{c^k-b^k}-x_A}{y_A}\cdot x_O+\frac{c^2-\left(\frac{c^k}{c^k-b^k}\right)^2}{2 y_A}$$
By plugging this into the conic equation, we get $0$ which finishes the proof.

For (3), we will show that the points lie on this conic section:
    \begin{equation}\label{eq:7}
        4  y_{O}^{2} y_{A}^{2}\left(p+l\right)^{2}
    \end{equation}
    $$+ x_{O}  y_{O} y_{A} \left(p+l\right)^{2} \left(8 x_{A} - 4\right)$$ 
    $$ + x_{O}^{2} \left(\left(p+l\right)^{2} \left(2 x_{A} - 1\right)^{2}- \left(p-l\right)^{2}\right) $$
    $$- 2y_{O} y_{A}\left(2 c^{2} \left(p+l\right)^{2}- \left(p-l\right)^{2}\right)$$
    $$+ x_{O} \left(- 2 c^{2} \left(p+l\right)^{2} \left(2 x_{A} - 1\right) + 2 x_{A} \left(p-l\right)^{2}\right)$$ 
    $$+c^{2} \left(c^{2} \left(p+l\right)^{2} - \left(p-l\right)^{2}\right)=0$$
Notice that this conic is exactly \eqref{eq:6} when we negate $l$. First we will calculate the coordinates of $O_{k,k+t}$ around $AM_kM'_{k+t}$, and once we do this, as $t$ can be negative, we have proven the fact for $AM'_{k}M_{k+t}$ :
$$x_O=\frac{\frac{c^{k} l}{ c^{k} l- b^{k} p} + \frac{c^{k}}{c^{k} + b^{k}}}{2}$$
$$y_O=\frac{\frac{c^{k}}{c^k+b^k }-x_A}{y_A}\cdot x_O+\frac{c^2-\left(\frac{c^k}{c^k+b^k}\right)^2}{2 y_A}$$

Now we prove respective corner cases for Theorem \ref{thm1}(2) and (3):
For the case where $AM'_k$ is the height in (2) the coordinates of $O_{k,k+t}$ are the same as for (1) due to minuses canceling out, so that case is already checked.

For the case where $AM_k$ is the height in (3), the coordinates of $O_{k,k+t}$ are:
$$y_O=\frac{y_A}{2}$$
$$x_O=\frac{x_A+\frac{l\cdot x_A}{(l+p)x_A - p}}{2}$$
Now we check that this point satisfies \eqref{eq:7} and therefore lies on the conic.

The corner case of zero span is already solved for (2) because the curves coincide. The case where $t=0$ is left out of the statement (3), although it still holds, we will not prove it here, but one could easily check that all the points lie on the line $ZV$.
Now we prove the first sentence of Theorem \ref{thm1}(5): For curves of type 1 and 2 it is trivial because triangle $\triangle AM_kM_{k+t}$(or $M'$) is the same as triangle $\triangle AM_{k+t}M_{k}$(or $M'$), first of which has span $+t$ and the second has span $-t$.

For curves of type 3, for some fixed $t$, $l$ and $p$ from equation \eqref{eq:7} will transform into $\frac{1}{l}$ and $\frac{1}{p}$ when we write formula of our conic section for span $-t$. If we write $(\frac{1}{p}\pm \frac{1}{l})^{2}$ as $(pl)^{-2}\cdot(l\pm p)^{2}=(pl)^{-2}\cdot(p\pm l)^{2}$, suddenly our polynomial becomes divisible by $pl^{-2}$ and since $pl\neq 0$, by dividing we get the polynomial for span $t$, which means the curve is identical.

Now we will prove Theorem1(4), and finally finish the proof of Theorem \ref{thm1}.

First, we will analyze point $Z$. We know that x-coordinate of $Z$ is zero. Now we will split this into two cases:
\begin{enumerate}
    \item $x_A=0$: Then point $Z$ has coordinates $Z \; (0,\frac{y_A}{2})$, and we want to prove that line $y=\frac{y_A}{2}$ is tangent to all of our conic sections
    \item $x_A \neq 0$: Line $AB$ has slope $k=\frac{y_A}{x_A}$. Then, perpendicular bisector of $AB$ has slope $k_s = \frac{-x_A}{y_A}$, and goes through point $S \; (\frac{x_A}{2},\frac{y_A}{2})$. We can easily calculate that the constant term of perpendicular bisector of $AB$, and also y-coordinate of point $Z$ is $n_s = \frac{x_{A}^2 + y_{A}^2}{2y_A}$. So, $Z \; (0, \frac{x_{A}^2 + y_{A}^2}{2y_A})$, and we want to show that line $y=-\frac{x_A}{y_A}\cdot x + \frac{x_{A}^2 + y_{A}^2}{2y_A}$ is tangent to all of our conic sections.
\end{enumerate}

First, if we prove claims for one of the points, for example, $Z$, the same will hold for the other because all we did was without loss of generality.

Both of these cases are proven by substituting the coordinates of $Z$ into \eqref{eq:6} (or \eqref{eq:7}) and getting 0, and then substituting it into the formula of a tangent to a respective conic given a point on the conic and getting the desired line formula. 

Therefore the proof of Theorem \ref{thm1} is finally finished.

\section{Proof of Theorem \ref{thm2}}
First, we observe the parabola $g$ with focus $A$ and directrix $BC$. It satisfies the theorem, as it can be found for span 0, so $g\in \mathbb{F}$. $\mathbb{F}$ contains only one parabola since a parabola touching two lines in two given points is by definition B\'ezier curve of order 2, which is unique.
\begin{lemma}\label{lem0}
Let $L$ be a point. The circle with radius $|LA|$ and center $L$ intersects with $BC$ if and only if $L$ is outside of $g$ or $L\in g$.
\end{lemma}
\begin{proof}
$L$ is not strictly inside of $g$ if and only if $|LA|$ is less then or equal to the distance between $L$ and  line $AB$. This is equivalent to the statement.
\end{proof}
\begin{lemma}\label{lem1}
Let us observe a Cartesian coordinate system with the x-axis being parallel to the directrix of $g$ and its focus lying on the y-axis and $g$ has the leading coefficient greater than 0. Let $L$ be a point inside of $g$ and $P\in g$ such that $L$ has y-coordinate not greater than the y-coordinate of $P$. Let $e$ be a ray from $P$ parallel to $LP$ but not containing $L$. For a point $S$ on $e$ let $S'$ be an intersection of $g$ with line perpendicular to $e$ in $S$ which is closer to $S$. We say that $SS'$ is \textbf{the value of $S$}. Then the value is strictly increasing with respect to $|PS|$.
\end{lemma}

\begin{proof}

Let $X$ and $Y$ be two distinct points on $e$, such that $|PX| < |PY|$. We define $X'$ and $Y'$ as in the statement of lemma (figure 4). Now, to prove lemma, we should prove that $x = |XX'| < |YY'| = y$. Let slopes of line $LP$, tangent to parabola at $X'$ and line $X'Y'$ be $k_1$, $k_2$ and $k_3$ respectively. First, we can observe that all of these slopes are well defined because neither of these lines is parallel to y-axis - line $LP$ can't be parallel to y-axis because then y-coordinate of $L$ would be greater than y-coordinate od $P$, tangent to the graph of quadratic function $y=ax^{2}+bx+c$ is never parallel to the y-axis, and line $X'Y'$ can't be parallel to y-axis because points $X'$ and $Y'$ are two different points on a graph of a function.

Line $LP$ must intersect parabola two times so point $P$ must be above line tangent to the parabola with slope $k_1$. Now we will assume, without loss of generality, that $k_1 \geq 0$. Point $X'$ is now also above line tangent to the parabola with slope $k_1$, and because both have non-negative x-coordinate, and derivative of this quadratic function is strictly increasing function, we know that $k_2 > k_1$. Since tangent to parabola intersects it only in one point, quadratic function with leading coefficient is increasing on positive numbers, and x-coordinate of $Y'$ is greater than x-coordinate of $X'$, we also know that line $X'Y'$ has a greater slope than tangent at $X'$ or $k_3 > k_2$. Now, because we know that projections of points $X'$ and $Y'$ on line $LP$ have greater x-coordinates than point of intersection of lines $LP$ and $X'Y'$, and $k_3 > k_2 \geq 0$, we can finally get to conclusion that $y > x$.

\begin{figure}[htbp]
  \centering

  \includegraphics[scale=0.4]{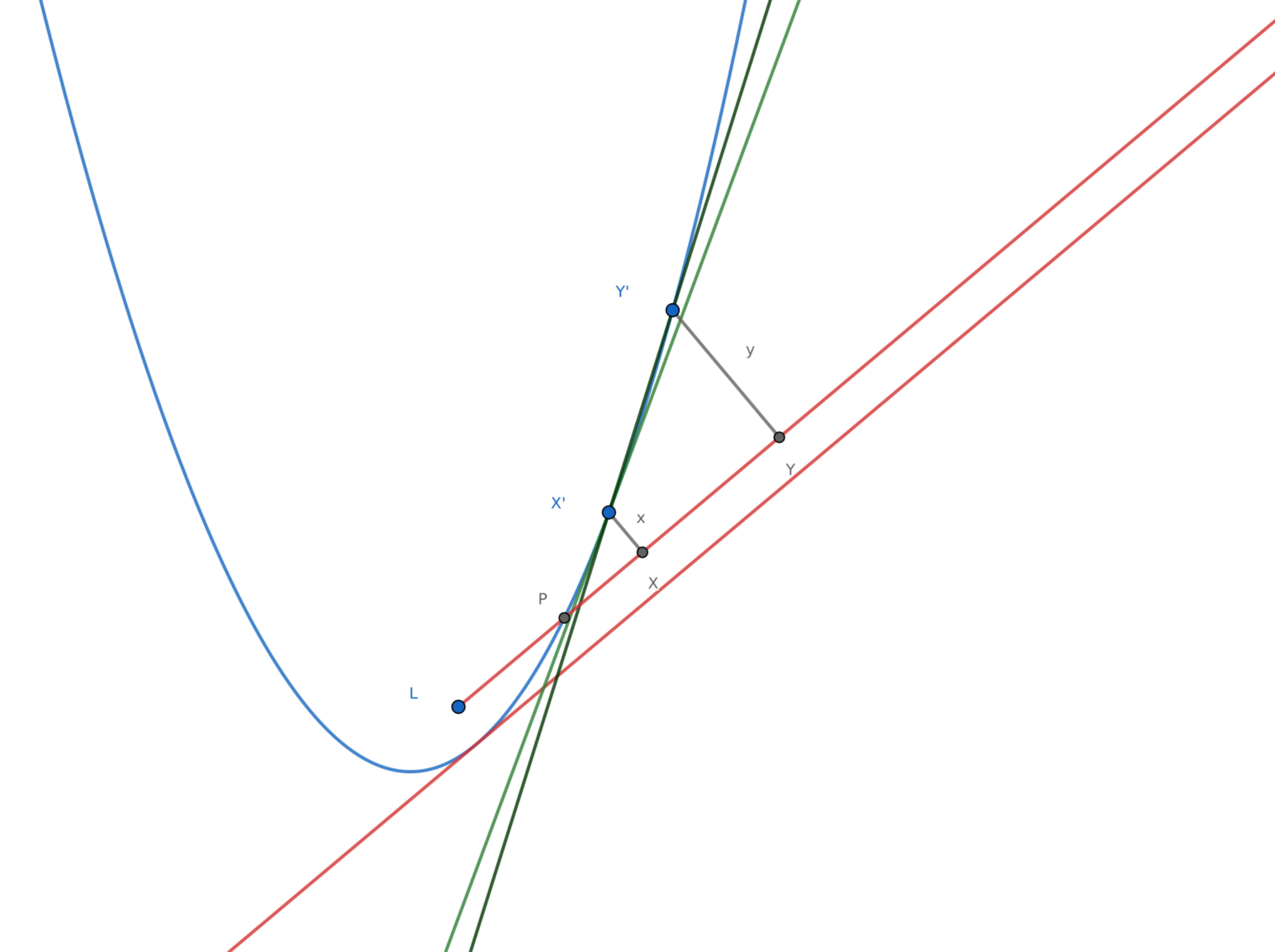}
  \caption{Lemma 4}
 \end{figure}
 
\end{proof}
\begin{lemma}\label{lem2}
If a hyperbola does not intersect with $g$, then no point of that hyperbola is inside of $g$.
\end{lemma}

\begin{proof}
We provide proof by contradiction. Let's assume there is a point $S$ of some hyperbola $u$ lying inside of $g$. Since each branch of a hyperbola is continuous, one whole branch $u_1$ is fully inside of $g$(there may be points lying on $g$ as well). Now we can observe its asymptotes. Let $l_1$ and $l_2$ be rays of the asymptotes of $u$ with their beginning in the center of hyperbola $L$ such that they are closer to $S$. We center a coordinate system such that the y-axis is the axis of $g$, the x-axis is the directrix of $g$ and $g$ has the leading coefficient greater than 0. Since $u$ is a non-degenerate hyperbola, at least one of the $l_1$,$l_2$ is not parallel with the y-axis. Without loss of generality, let that ray be $l_1$. Then $l_1$ intersects $g$ in some point $P$. Then if $L$ has y-coordinate greater than $P$, the ray contains a point with arbitrarily small y-coordinate as we go up the ray, but the points of hyperbola get arbitrarily close to $l_1$ as we go up the ray by the definition of the asymptotes, and therefore we can find a point of $u_1$ which is outside of $g$ since $g$ has its lowest point, which is a contradiction since $u_1$ is continuous. Therefore $L$ must be under $P$. We know that as we go up the asymptote $u_1$, the points of hyperbola get strictly closer to $u_1$, but by lemma \ref{lem1}, we know that as we go up the ray, the points of $g$ get strictly further away from it, and so there will be a point of $u_1$ which is closer to $l_1$ than it is to its closest point on $g$, from which follows it is outside of $g$. Again, we get a contradiction because of the continuity.
\end{proof}

Now we go back to the proof of Theorem \ref{thm2}. By lemma \ref{lem0} we know that every point which is outside of the parabola has representation in Theorem \ref{thm1}. Only the points in the parabola do not, but by lemma \ref{lem2} we know that none of these points belong to any parabola or hyperbola of $\mathbb{F}$. Therefore the proof of Theorem \ref{thm2} is completed.

\section{Applications}

Now we will introduce some well-known already solved problems. These problems are generally trivially solved using Pascal's theorem and some Euclidean geometry, but we will provide alternative solutions using our recently proved theorems. Furthermore, none of these problems rely on metric properties of conic sections in the Euclidean plane, so we justifiably hope that it has some real applications, given that even powerful ideas from projective geometry do not generally solve metric problems involving conics.

Just before we start, we introduce one final tool.
\begin{definition}
Let $p$ and $q$ be distinct lines intersecting at $O$ and let $Z$ and $V$ be points on $p$ and $q$ respectively, $Z,V\neq O, |OZ|\neq |OV|$. The pencil of conics touching $p$ and $q$ in $Z$ and $V$ is said to be an \textit{$\chi$-type pencil}. $p$ and $q$ are then sides of the pencil, $O$ is the vertex of the pencil, and $Z, V$ are contact points. Line $ZV$ is called the critical line of the pencil. 
\end{definition}
\begin{remark}
By Theorem \ref{thm2}, every $\chi$-type pencil contains a unique parabola.
\end{remark}
\begin{definition}
Let $\mathbb{F}$ be an $\chi$-type pencil.The sides of $F$ divide the plane into 4 regions, we say that the region containing parabola is $R_1$, and we label the other three regions in clockwise order $R_2,R_3$ and $R_4$. Furthermore, we divide $R_1$ without the parabola into 4 regions: $U_1$ is the set of all points inside of parabola, $U_2$ is the set of all points $K$ such that $K$ and $C$ lie on opposite sides of the line $BZ$, $U_3$ is the set of all points $K$ such that $K$ and $B$ lie on opposite sides of the line $CV$. Let $g$ be the parabola of $\mathbb{F}$ , the $U_4$ is the remaining region, $R_1\backslash (U_1\cup U_2 \cup U_3 \cup g)$
\end{definition}
\begin{proposition} \label{prop}
Let $\mathbb{F}$ be an $\chi$-type pencil, let $g$ be its parabola and let $A$ its focus. Let $B$ and $C$ be projections of $Z$ and $V$ onto the directrix of $g$, respectively.Let $m$ be the segment $BC$ and let $m'$ be line $BC$ without $m$. Let $E$ be a point and let $k$ be the circle centered at $E$ with radius $|EA|$. Then the following holds:
\begin{enumerate}
    \item $E\in R_3$ if and only if $k$ contains $m'$
    \item $E\in (R_2\cup R_4)$ if and only if $k$ intersects with both $m$ and $m'$.
    \item If $E\in U_1$, then $k$ has no intersection points with $m$ nor $m'$
    \item If $E\in g$, then $k$ has one point of intersection with $m\cup m'$
    \item If $E\in (U_2\cup U_3)$, then $k$ has two points of intersection with $m'$
    \item If $E\in U_4$, then $k$ has two points of intersection with $m$.
\end{enumerate}
\end{proposition}
\begin{proof}

\begin{enumerate}
   
    \item \label{it2} Since the side of $\mathbb{F}$ containing $Z$ is the perpendicular bisector of $AB$, and by definition of $R_3$, $B$ and $E$ lie on the same side of it, $|EB|< |EA|$ and therefore $k$ contains $B$. Analogously, $k$ contains $C$, and by the convexity of the circle, $m$ is fully inside of $k$. By lemma \ref{lem0}, $k$ intersects $m\cup m'$ but it contains $m$, so both points of intersection lie in $m'$. The opposite is also true, if both intersection points lie outside of $m$, then the segment containing these points also contains $B$ and $C$, and therefore the circle contains them as well, which is equivalent to $B$ and $E$ being on the same side of $OZ$, where $O$ is the center of $\mathbb{F}$ and $C$ and $E$ lying on the same side of $OB$, which is equivalent to $E$ being in $R_3$  
    \item \label{it1} Analogously to \ref{it2}, we know that $E\in (R_2\cup R_4) \Leftrightarrow( |AE|>|EB| $ and ($|AE|<|EC|)$ and $|AE|<|AB|$ or $|AE|>|EC|$) . This is equivalent to the circle containing exactly one of the points $B, C$, which is equivalent to $k$ intersects both $m$ and $m'$.   
    \item This follows from lemma $\ref{lem0}$
    \item This is trivial since it follows from the definition of a parabola
    \item \label{it3} We already proved that only points $E$ for which circle $k(E,|EA|)$ intersects $m'$ two times and on different sides of point $B$ are points from $R_3$ (\ref{it2}), and that only points $E$ for which circle $k(E,|EA|)$ intersects $m$ once and $m'$ once are points from $R_2 \cup R_4$ (\ref{it1}). This leaves only two options for points $E$ from $U_2$ ($U_3$ case is done analogously) - circle $k(E,|EA|)$ either intersects $m'$ two times and intersection points are on the same side of point $B$ or it intersects $m$ two times. We will prove that the second case is impossible by contradiction. Assume that circle $k$ intersects $m$ two times. Then, the perpendicular bisector of a segment which endpoints are two intersection points would also intersect $m$. This is impossible, because $E$ lies on that perpendicular bisector, and points $E$ and $C$ should lie on different sides of the line $BZ$. Contradiction. That means that circle $k$ intersects $m'$ two times. (Figure 6)
    \item Let $E$ be a point from $U_4$, and $k(E,|EA|)$ be a circle. Similarly as in (\ref{it3}) we know that there are only two cases, circle $k$ either intersects $m$ two times, or it intersects $m'$ two times, and these two points are on the same side of point $B$. We will now prove, again by contradiction, that the second case is impossible. Assume that circle $k$ intersects $m'$ two times, and, without loss of generality, let these two points be on the different side of point $B$ than point $C$. Similarly as earlier, we now know that point $E$ must lie on line $r$ perpendicular to line $BC$, such that the intersection of $r$ and $m'$ and point $C$ are on the different sides of point $B$. This would mean that segment $CE$ does intersect with line $BZ$, which is impossible because point $E$ isn't from $U_2$. Contradiction. That means that circle $k$ intersects $m$ twice.
\end{enumerate}
\end{proof}

\begin{figure}[htbp]
  \centering

  \includegraphics[scale=0.4]{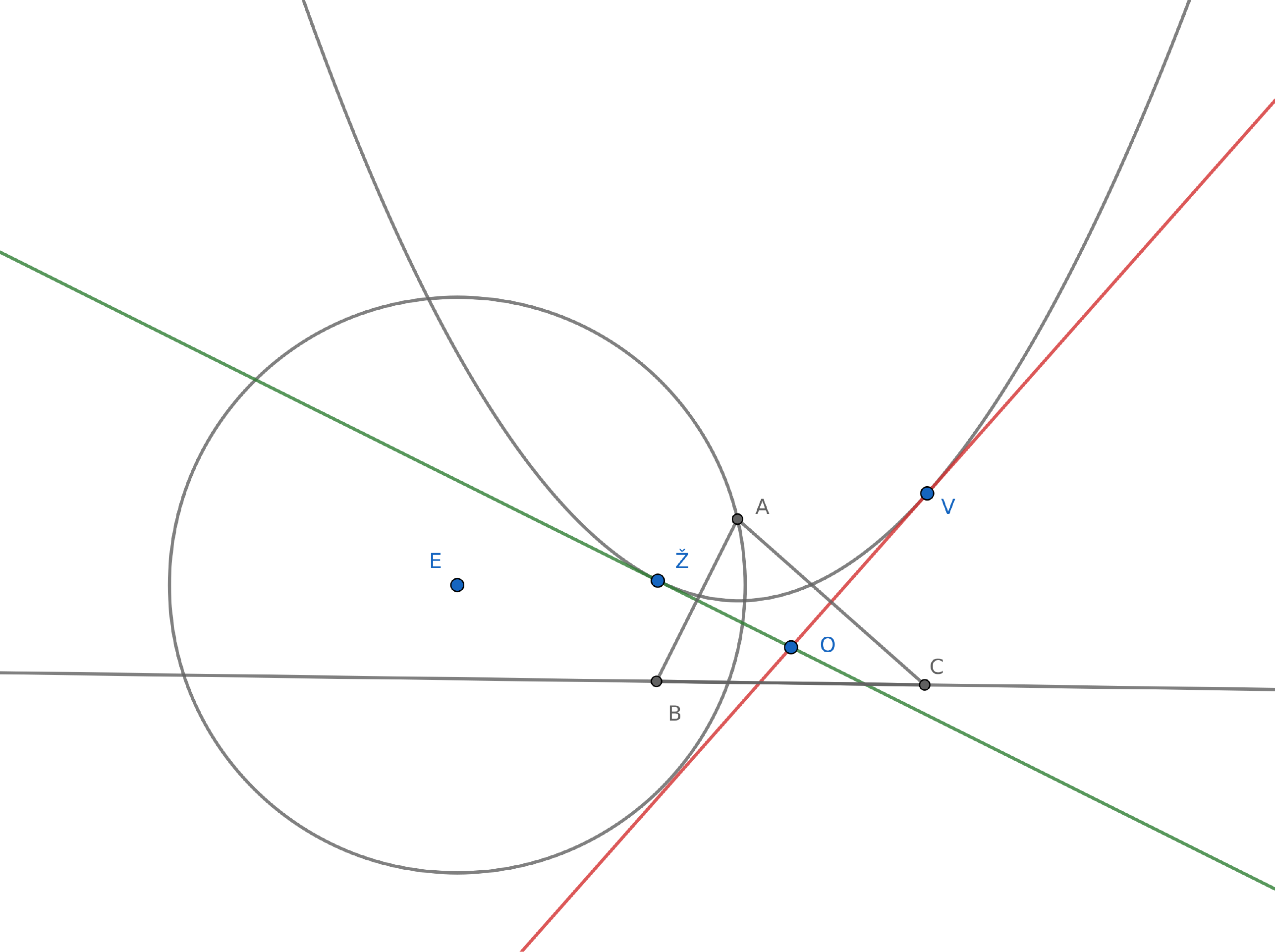}
  \caption{Proposition 6(2)}
 \end{figure}
 
 \begin{figure}[htbp]
  \centering

  \includegraphics[scale=0.4]{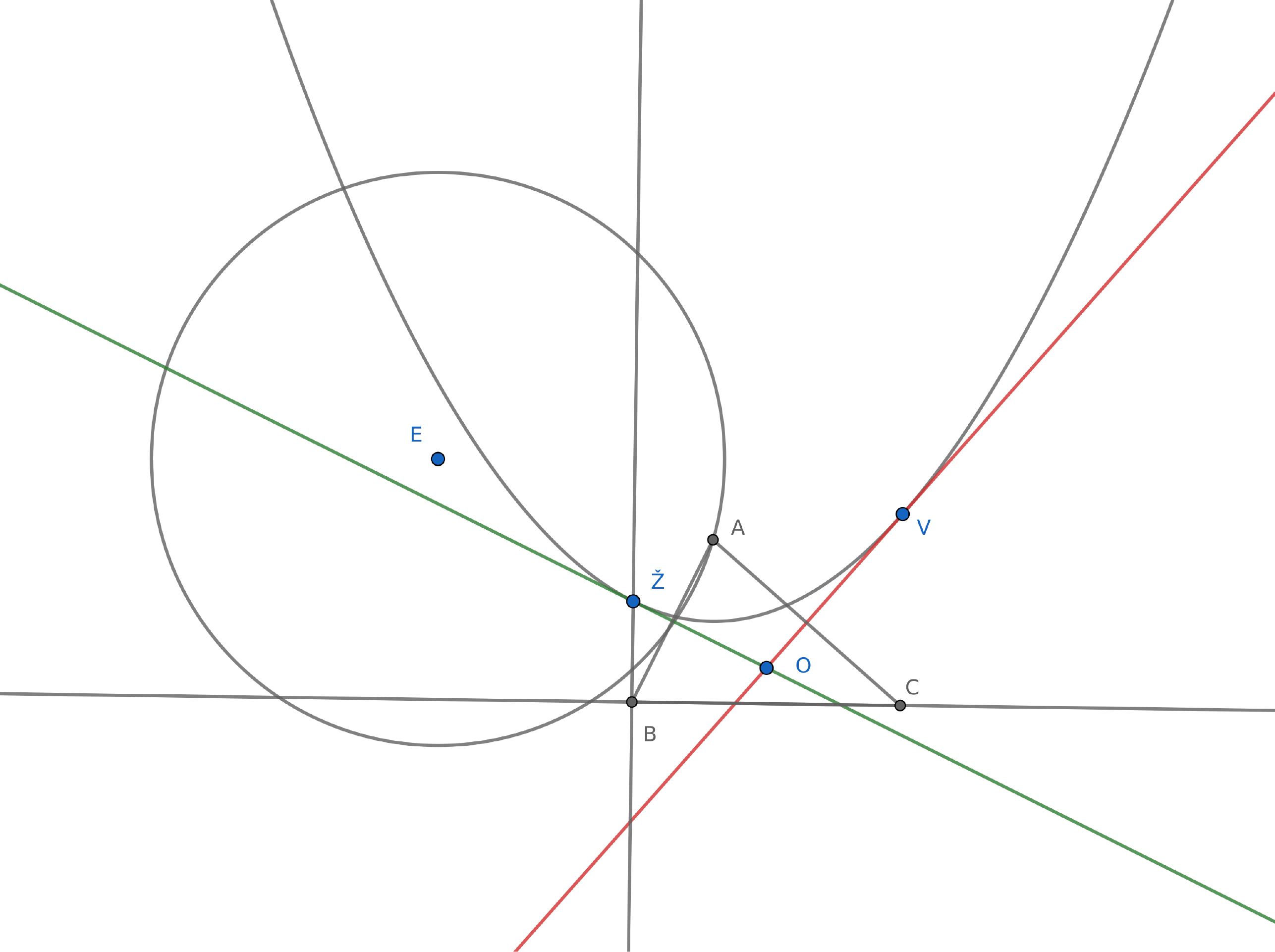}
  \caption{Proposition 6(5)}
 \end{figure}

\begin{theorem}[Classification of conic sections in $\chi$-type pencils]
Let $\mathbb{F}$ be an $\chi$-type pencil, $W$ be a point and $u$ be a conic in $\mathbb{F}$  passing through $W$. Then the following holds:
\begin{enumerate}
    \item If $W$ is on the parabola, then $u$ coincides with it
    \item\label{el} $u$ is an ellipse if and only if $W$ is inside of the parabola and not on the segment $ZV$
    \item if $W$ belongs to a side of the pencil, then $u$ is a degenerate hyperbola, $p\cup q$
    \item if $W$ belongs to $(R_2\cup R_4)\backslash ZV$, then $u$ is a hyperbola. We call this a \textit{side hyperbola}
    \item if $W$ belongs to $R_1\cup R_3$ and $W$ is strictly outside of the parabola, then $u$ is a hyperbola. We call this a \textit{straight hyperbola}
\end{enumerate}
\end{theorem}
\begin{proof}

\begin{enumerate}
\item We know the parabola is unique.
\item If $u$ is inside of parabola it has to be an ellipse since by lemma \ref{lem2} it cannot be a hyperbola. If $u$ is an ellipse, since it is convex and contains $Z$ and $V$, it also contains the whole segment $ZV$. The same holds for the parabola, and therefore they have common interior points. Since they do not intersect and the parabola is infinite, the parabola contains $u$.
\item $u$ touches $p$ in $Z$ and on an additional point, so it has to contains $p$ and $q$ analogously. Since it is a conic, it has to be degenerate conic $p\cup q$
\item \label{hiptriv}$u$ can neither be an ellipse(by (\ref{el})), nor a parabola, so it has to be a hyperbola
\item Analogously to \ref{hiptriv}
\end{enumerate}
\end{proof}

\begin{figure}[htbp]
  \centering

  \includegraphics[scale=0.4]{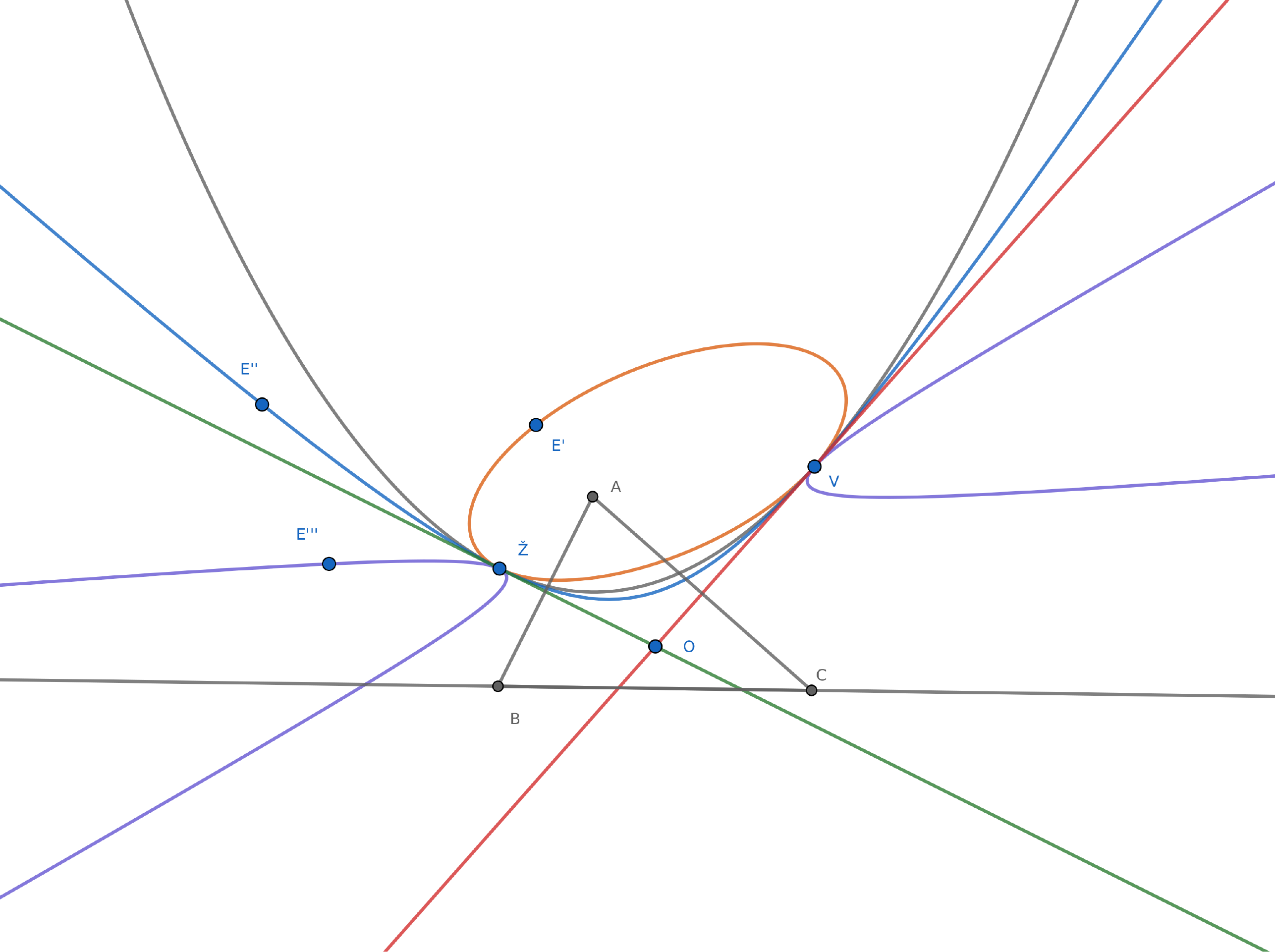}
  \caption{Classification theorem}
 \end{figure}

Now that we have all the theory, we can ask some natural questions:
\begin{problem}
Given points $X$ and $Y$ and let $\mathbb{F}$ be an $\chi$-type pencil such that neither of points $X$ nor $Y$ lie on the critical line of $\mathbb{F}$ and neither of $X$, $Y$ lie on the sides of $\mathbb{F}$. When do $X$ and $Y$ lie on the same conic in $\mathbb{F}$?
\end{problem}

\begin{solution}
First, we will use the classification theorem to eliminate all cases in which the curves determined by $X$ and $Y$ are of a different type. The remaining cases are:
\begin{enumerate}
    \item Both $X$ and $Y$ lie on the parabola:
    
        This case is trivial since parabola is unique in $\mathbb{F}$.
  
    \item Both $X$ and $Y$ lie in $R_2\cup R_4$ 
    
        By proposition \ref{prop} we know that circles $k_1 (X,|XA|)$ and $k_2 (Y,|YA|)$ intersect line $BC$ once in $m$ and once in $m'$. Let circle $k_1$ intersects with $m$ at point $X_1$ and with $m'$ at point $X_2$, and circle $k_2$ intersects with $m$ at point $Y_1$ and with $m'$ at point $Y_2$. We will now look at the ratios, $x_1 = \frac{|BX_1|}{|X_1C|}$, $x_2=\frac{|BX_2|}{|X_2C|}$, $y_1 = \frac{|BY_1|}{|Y_1C|}$, $y_2 = \frac{|BY_2|}{|Y_2C|}$. By \ref{thm1}, we know that points $X$ and $Y$ lie on the same conic section from $\mathbb{F}$ if and only if $| log_{\frac{c}{b}}x_1 - log_{\frac{c}{b}}x_2 | =  | log_{\frac{c}{b}}x_1 - log_{\frac{c}{b}}x_2 |$, which is equivalent to: $\frac{x_1}{x_2} \cdot \frac{y_1}{y_2} = 1$ or $\frac{x_1}{x_2} \cdot \frac{y_2}{y_1} = 1$. We can check both of these, and see if one of these two equations is true, which would tell us that points $X$ and $Y$ are on the same conic. If that isn't the case, then they don't lie on the same conic.
        
        \item Both $X$ and $Y$ lie in $U_2\cup U_3\cup U_4\cup R_3$:
    
        By proposition \ref{prop}, we know that the circle $k_1(X,|XA|)$ intersects either $m$ or $m'$ twice. The same holds for $k_2(Y,|YA|)$. Therefore the curves $u_1$ and $u_2$ containing $X$ and $Y$ respectively , by Theorem \ref{thm1}, both belong to type (1)(or (2), but it was proven that these are identical). The problem becomes equivalent to checking if $Sp(u_1)$ is equal to $Sp(u_2)$. The span is calculated analogously to the previous case
    \item Both $X$ and $Y$ lie in $U_1$
    
    We know these are ellipses by the classification theorem. To do this case we will use polar/pole transformation with respect to the parabola. If we can find a point that lies on the dual conic with respect to parabola for both $X$ and $Y$, then we have successfully reduced the problem to the previous two cases. For this, we will use the following lemma:
    
    \begin{lemma}
        Let $\mathbb{F}$ be a $\chi$-type pencil. Let $u\in \mathbb{F}$ be an ellipse inside of the parabola. and let $P$ be a point on it. Let $E$ and $F$ ($E\neq F$) be intersections of the tangent to $u$ in $P$ with the parabola, and let $G$ be the intersection of tangents to the parabola in $E$ and $F$. Then $P$, $O$, and $G$ are collinear. 
    \end{lemma}
    
    \begin{figure}[htbp]
  \centering

  \includegraphics[scale=0.4]{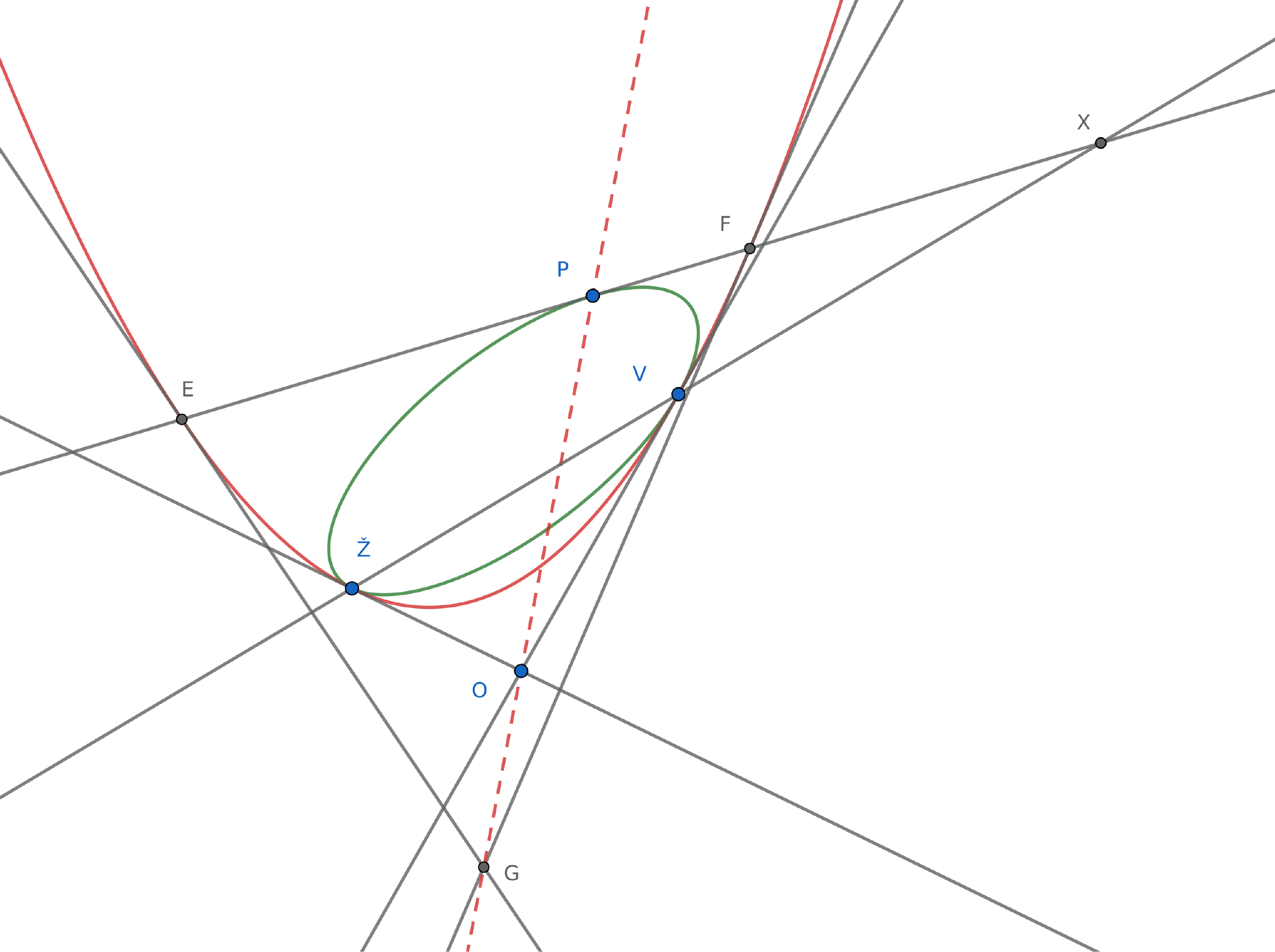}
  \caption{Lemma 5.3}
 \end{figure}
 
    \begin{proof}
    By definition, polar of $O$ with respect to the parabola is line $ZV$. Let $X$ be the intersection of lines $ZV$ and $EF$. Since $X$ lies on the polar of $O$, $O$ lies on the polar of $X$. Analogously, $G$ also lies on the polar of $X$. By proposition 266 from \cite{3}, we know that $(E,P;F,X)=-1$ i.e $X$ and $P$ are harmonic conjugates. We will now prove that the cross-ratio of the 4 points of intersection of any line through $X$ with the parabola and the polar of $X$ is -1. Note that once it is true, we know that every point has exactly one harmonic conjugate, so $P$ will also lie on the polar and the proof is finished. To prove this, we will use a projective transformation. By 1. property from section 2.2 in \cite{2}, all quadrilaterals are projectively equivalent, by 2. , they preserve cross-ratio on every line, and by proposition 7.2.4 from \cite{4} they also preserve the order of the curve. Let $L$ and $R$ be different points on the parabola such that $XL$ and $XR$ are tangents.Line $LR$  is then the polar of $X$ with respect to parabola (by definition). Let $t$ be any line containing $X$ such that it intersects the parabola and $X$'s polar and let the intersections be $F'$ and $P'$ and $E'$ such that $B(X,F',P',E')$. We then define a projective mapping by fixing $R$ and $L$ and sending $E'$ to point $E''$ equidistant from $R$ and $L$ such that $\angle RE''L$ is the right angle (there are two options of which we pick further from $X$). We map $X$ to the point at infinity on the perpendicular bisector of $RL$. Then the parabola maps to a circle with both maps of $LR$ and $E'F'$ being its perimeters. The centre of the circle is $P'$ and so the images are harmonic conjugates, which finishes the proof.
    \end{proof}
    Note that if the conditions of this lemma are met, then we can always find a point from the dual conic since it is equivalent to saying that the intersection of the line containing $O$ and $P$ and the polar to $P$ with respect to the parabola always intersect on the dual conic. Now we only have to deal with the case in which these two lines do not intersect. The only corner cases turn out to be points such that their tangent to the ellipse is parallel to the axis of the parabola, so there are only two cases to handle: if both $X$ and $Y$ are problematic, their midpoint must lie on the perpendicular bisector of $BC$, and that is also sufficient, and the other is when only one point is problematic. We use proposition XXI from \cite{1} to find the center of the ellipse and after that another non-problematic point by reflecting the one we already have about the center of the ellipse. This case is then reduced to the already solved one.
\end{enumerate}

\end{solution}

\begin{remark}
Using \cite{1}, it is possible to use the classification theorem along with other main results to do various Euclidean constructions involving conics and finding appropriate points on them using ruler and compass only. Even though it is possible, the constructions are far from short and are not the focus of this paper.
\end{remark}

\section{Further research}
We have noticed that the dual conic to a side hyperbola over the parabola in an $\chi$-type pencil is another side hyperbola. This leads us to wonder if they are related because we wanted to give span to every ellipse in the most natural manner possible, and we believe dual curves have something in common.
\begin{problem}
Given a side hyperbola $g_1$. Let $g_2$ be a dual to $g_1$ over the parabola. Are $Sp(g_1)$ and $Sp(g_2)$ related?
\end{problem}
\begin{problem}
Let $C(k)$ be the intersection of $AM_a(k)$,$BM_b(k)$ and $CM_c(k)$ (By Ceva's theorem, these points exist). What is the locus of $C(k)$ for $k\in \mathbb{R}$?
\end{problem}
\begin{problem}\label{probelip}Given a triangle $\triangle ABC$. Find the set of all points $D$ such that $A,B,C,O_A,O_B,O_C$ lie on a conic, where $O_A,O_B$ and $O_C$ are the cicumcenters of $\triangle BCD, \triangle ACD , \triangle ABD$, respectively. 
\end{problem}
\begin{conjecture}
The set of all interior points of the triangle in problem \ref{probelip} is an ellipse.
\end{conjecture}
\section*{Acknowledgement} Authors owe the highest gratitude to dr Đorđe Baralić for mentoring this project and a lot of helpful advice and discussions. Special thanks to Marko Milenković and professor Miloš Milosavljević for their contribution to an early stage of the paper.


\end{document}